\documentclass[12pt,leqno]{article}
\usepackage[utf8,latin1]{inputenc}
\usepackage[latin1]{inputenc}
\usepackage[utf8]{inputenc}
\usepackage{makeidx}

\usepackage{verbatim}
\usepackage{graphicx}
\usepackage{amsmath,amssymb}
\usepackage{color}
\usepackage{adjustbox}
\usepackage{mathtools}
\usepackage{mleftright}

\usepackage{xfrac}
\usepackage{amsthm}
\usepackage{amssymb}
\usepackage{wasysym}
\usepackage{tikz}
\usetikzlibrary{matrix}
\usepackage{ae}
\usepackage[all,cmtip]{xy}
\usepackage{mathtools}
\usepackage{enumerate}
\usepackage[new]{old-arrows}
\usepackage{fancyhdr}
\usepackage{blkarray}
\usepackage[strict]{chngpage}
\usepackage{mathdots}

\newcommand{\codim}{\operatorname{codim}}
\newcommand{\diag}{\operatorname{diag}}
\newcommand{\reg}{\operatorname{reg}}

\newcommand{\charact}{\operatorname{char}}
\newcommand{\im}{\operatorname{Im}}
\newcommand{\Irr}{\operatorname{Irr}}
\newcommand{\Ind}{\operatorname{Ind}}

\newcommand{\Id}{\operatorname{Id}}

\newcommand{\N}{\mathbb{N}}

\newtheorem{teorema}{Teorema}[section]
\theoremstyle{definition}

\newtheorem{definition}[teorema]{Definition}
\newtheorem{lemma}[teorema]{Lemma}
\newtheorem{proposition}[teorema]{Proposition}
\newtheorem{proposition/definition}[teorema]{Proposition/Definition}
\newtheorem{remark}[teorema]{Remark}
\theoremstyle{remark}

\title{Jordan classes and Lusztig strata in disconnected reductive groups}
\author{Martina Costa Cesari}

\author{Martina Costa Cesari\\
	Dipartimento di Matematica ``Tullio Levi-Civita''\\
	Torre Archimede - via Trieste 63 - 35121 Padova - Italy\\
	email: martina.costacesari@math.unipd.it}
\date{}

\begin{document}
	\maketitle
	
	\begin{abstract}
		Let $G$ be a  non-connected reductive algebraic group over an algebraically closed field $\mathbb{K}$ and let $D$ be a connected component of $G$. We investigate Jordan classes of $D$ and we obtain a description of the regular part of the closure of a Jordan class in terms of induction of $G^{\circ}$-orbits. We use this result to show that Lusztig strata in a non-connected reductive algebraic group are locally closed.
	\end{abstract}
	\begin{section}*{Introduction}
		Reductive non-connected groups appear frequently in the study of algebraic groups, for example as centralizers of semisimple elements in non-simply connected semisimple groups. Let $G$ be a non-connected reductive algebriac group over an algebraically closed field $\mathbb{K}$ of arbitrary characteristic and 
		let $D$ be a connected component of $G$. In \cite{lusztig2020strata} G. Lusztig defined a partition of $D$ in finitely many strata, generalizing the partition defined in \cite{Lusztig2015} for $G^{\circ}$. 
		There is an action of $G/G^{\circ}$ on the Weyl group $W$: we will denote by $W^{D}$ the fixed point set for the action of $D$ on $W$. The subgroup $W^D$ is a Weyl group.
		Using a variant of the Springer's reprsentation, it possible to define a map $E$ from $D$ to the set of irreducible representations of $W^D$.
		Lusztig's strata are the fibers of $E$.
		
		It was proved in \cite{carno} that the strata in $G^{\circ}$ are locally closed.  G. Lusztig suggested in \cite{lusztig2020strata} that also the strata of $D$ are locally closed.
		In this paper we prove this assertion. 	From the definition of strata it is not immediate that they have geometrical properties, so thanks to this result, we can treat them as geometrical objects. For this purpose, we study the partition of $D$ into Jordan classes, i.e. the equivalence classes defined in \cite{charactersheafI}. Strata of $D$ are union of finitely many Jordan classes of elements whose $G^{\circ}$-orbits have same dimension \cite{lusztig2020strata}.
		We show that a stratum is a union of the regular part of the closure of Jordan classes. For this reason we need to study Jordan classes and their closures.\\
		Jordan classes for $G^{\circ}$ were defined in  \cite{intersectioncohom}. Our aim is to generalize some properties of Jordan classes of $G^{\circ}$ to Jordan classes of $G$. We describe a procedure of induction of an orbit from a connected component of the normalizer of a Levi subgroup of $G^{\circ}$ to a $G^{\circ}$-orbit in $D$. This allows us to investigate the closure and the regular closure of a Jordan class. In particular, similarly to the connected case (\cite{CarnovaleEsposito}), the regular closure of a Jordan class is a union of induced $G^{\circ}$-orbits.		\\
		
		The paper is structured as follows.
		In the first section we give the definition of Jordan classes and recall relevant properties. Furthermore we describe the induction of $G^{\circ}$-orbits in $D$, and investigate the regular closure of a Jordan class.
		In the second section, using the description of regular closure of Jordan classes, we prove that a stratum $X\in D$ is locally closed.\\
		The last section is devoted to examples: we describe the Jordan classes of $G=SL(n)\rtimes \langle\tau\rangle$, where $\tau$ is the exterior automorphism reversing the Dynkin diagram of type $A_{n-1}$  and $\charact \mathbb{K}=2$.

	\end{section}

	\begin{section}*{Notation}
		Let $G$ be a reductive algebraic group, not necessarily connected, and let $G^{\circ}$ be its identity connected component. If $x$ is an element of $G$, and $H$ a closed subgroup of $G$, then
		
		\begin{itemize}
			\item $H^{\circ}$ is the identity component of $H$;
			\item $N_GH$ is the normalizer of $H$ in $G$;
			\item $Z(H)$ is the center of $H$;
			\item $x_s$, $x_u$ are, respectively, the semisimple and the unipotent part of $x$;
			\item if $G$ acts on a variety $X$ and $x\in X$, we denote by $\operatorname{Stab}(x)$ the stabilizer of $x$ in $G$, i.e. $\operatorname{Stab}(x)=\{h\in G \ | \ h\cdot x=x\}$;
			when $X=G$ and the action is conjugation, we will denote  by $G^x$ or with $C_G(x)$ the centralizer of $x$ in $G$. If $H$ is a subgroup of $G$, we will denote by $H^{x}$ the centralizer of $x$ in $H$;
			\item If $X$ is a $G$-variety, then $X_{(d)}=\{x\in X \ | \ \dim (G\cdot x)=d\}$. Let $Y\subseteq X$ and let $m$ be the maximum integer such that $Y\cap X_{(m)}\neq \emptyset$. We will  denote by $Y^{\text{reg}}$ the set of regular elements of $Y$, i.e., the elements $y\in Y$ such that $\dim( G\cdot y) =m$;
			\item $D$ is a connected component of $G$.
			
		\end{itemize}
	\end{section}
	
	\begin{section}{Jordan classes}
		In this section, we recall the definition of Jordan classes and their properties \cite{charactersheafI} and study the closure and the regular closure of a Jordan class. In order to do that, we will describe the induction of $G^{\circ}$-orbits in $D$.
		
		\begin{subsection}{Preliminars on Jordan classes and isolated elements}\label{preliminars}
			In this section, we give the definition of a Jordan class and an isolated Jordan class.
			Furthermore, we recall the analogue of parabolic subgroups in the disconnected case, namely the normalizers in $G$ of a parabolic subgroup of $G^{\circ}$, and their properties.\\

			Let $a$ be an element of $G$ with Jordan decomposition $a=a_sa_u$. We set 
			$$
			T(a)=(Z((C_{G}(a_s))^{\circ})\cap C_{G}(a_u))^{\circ}=(Z((C_{G^{\circ}}(a_s))^{\circ})\cap C_{G^{\circ}}(a_u))^{\circ}
			$$
			
			We consider the equivalence relation on $G$:
			$$a \sim h \ \text{if} \  \exists x\in G^{\circ} \  \text{such that} \  T(xhx^{-1})=T(a) \ \text{and}  \ xhx^{-1}a^{-1}\in T(a).$$
			A Jordan class is an equivalence class for this relation, and we denote the Jordan class containing $a$ by $J(a)$.
			
			By \cite[Corollary 9.4]{steinberg}, $C_{G^{\circ}}(a_s)^{\circ}$ is a reductive group, so $Z(C_{G^{\circ}}(a_s))^{\circ}$ is a torus. Hence $T(a)$ is a closed connected subgroup of a torus, so $T(a)$ is a torus. \\

			\begin{remark}\label{sameunipotentpart}
				By construction elements in a Jordan class, up to conjugation, have the same unipotent part. In fact, if $a\sim h$, then, up to conjugation, $h\in T(a)a$ with $T(h)=T(a)$ so $h=za$, where $z\in T(h)$ is a semisimple element, because $T(h)$ is a torus, and it commutes with $h_s,h_u,a_s$ and $a_u$, so $h_s=za_s$ and $h_u=a_u$.
				
				\begin{remark}\label{regpartofsemisimple}
					Let $h\in T(a)a$. By Remark \ref{sameunipotentpart}, $h_s=za_s\in T(a)a_s$, with $z\in T(a)\subset Z(C_G(a_s)^{\circ})$. Let  $x\in C_G(a_s)^{\circ}$. Then $z$ commutes with $x$, hence $h_s=za_s=zxa_sx^{-1}=xza_sx^{-1}=xh_sx^{-1}$. 
					Thus $C_G(a_s)^{\circ}\subset C_G(h_s)^{\circ}$. So $\dim(C_G(a_s)^{\circ}) \leq \dim(C_G(h_s)^{\circ})$ for all $h\in T(a)a_s$. Therefore
					$$(T(a)a_s)^{\reg}=\{h_s\in T(a)a_s \ | \ C_G(h_s)^{\circ}=C_G(a_s)^{\circ} \}.$$
				\end{remark}
				\begin{remark}\label{centsemisuguali}
					Let $J$ be a Jordan class, and $a=a_sa_u, \ h=h_sh_u\in J$. Then, up to conjugation,  $C_G(h_s)^{\circ}=C_G(a_s)^{\circ}$.
					Indeed, by definition of Jordan classes, up to conjugation,  $h\in T(a)a$ and $a\in T(h)h$, so by Remark \ref{regpartofsemisimple}, $C_G(a_s)^{\circ}= C_G(h_s)^{\circ}$.	
				
					Also, if $a\sim h$ then $C_G(a)^{\circ}$ and $C_G(h)^{\circ}$ are $G^{\circ}$-conjugated (see \cite[Lemma 3.4]{charactersheafI}). As a consequence, $J\subset G_{(d)}$ for some $d$. 
				\end{remark}

				By Remark \ref{sameunipotentpart} and Remark \ref{centsemisuguali} the semisimple parts of the elements in the Jordan class $J(a)$ are contained in $G^{\circ}\cdot(T(a)a_s)^{\text{reg}}.$\\
				
				We denote 
				\begin{equation}\label{Jordanreg}
					(T(a)a)^{\bullet}=\{h\in T(a)a \ | \ T(h)=T(a) \}.
				\end{equation}

				\begin{lemma}\label{regpart}
					With notation as above there holds
					$$
					(T(a)a)^{\bullet}=(T(a)a_s)^{\text{reg}}a_u
					$$
				\end{lemma}
				\begin{proof}
					Let $h\in T(a)a$ such that $T(h)=T(a)$. Thus $h=za_sa_u$ with $z\in T(a)$. By definition of $T(a)$, $z$ is semisimple and commutes with $a_s$ and $a_u$, so $h_s=za_s$ and $h_u=a_u$.
					By Remark \ref{centsemisuguali}, $C_G(h_s)^{\circ}=C_G(a_s)^{\circ}$, so, by Remark \ref{regpartofsemisimple}, $x\in (T(a)a_s)^{\text{reg}}a_u$.\\
					Conversely, let $y\in (T(a)a_s)^{\text{reg}}a_u$. Then $y_s=xa_s$ and $y_u=h_u$ with $C_G(xa_s)^{\circ}=C_G(a_s)^{\circ}$. So 
					
					$$
					Z((C_{G}(xa_s)^{\circ})=Z((C_{G}(a_s)^{\circ}),
					$$
					thus
					$$
					T(y)=(Z((C_{G}(xa_s))^{\circ})\cap C_{G}(a_u))^{\circ}=(Z((C_{G}(a_s))^{\circ})\cap C_{G}(a_u))^{\circ}=T(a).
					$$
					and  $y\in (T(a)a)^{\bullet}$.
				\end{proof}
				By Lemma \ref{regpart}, we have
				\begin{equation*}\label{semisimplepartreg}
					J(a)=G^{\circ}\cdot ((T(a)a_s)^{\text{reg}}a_u).
				\end{equation*}
				
				We set
				$$L(a) = C_{G^{\circ}}(T(a)).$$
				
				Since $L(a)$ is a centralizer of the torus $T(a)$, it is a Levi subgroup of $G^{\circ}$. Hence, there exists a parabolic subgroup $P$ of $G^{\circ}$ that has $L(a)$ as a Levi subgroup. Furthermore one may choose $P$ so that $a\in N_GP$ (\cite[2.1 (a)]{charactersheafI}).
			\end{remark}
			
			By \cite[2.2]{charactersheafI} the following conditions for $a \in G$ are equivalent:
			\begin{enumerate}[(i)]
				\item \label{i} $L(a)=G^{\circ}$;
				\item $T(a) \subset Z(G^{\circ})$
				\item $T(a)=Z(G^{\circ})\cap C_G(a)$;
				\item there is no proper parabolic subgroup $P$ of $G^{\circ}$ with a Levi subgroup $L$ such that $a \in N_{G} P \cap$ $N_{G} L$ and $C_{G}\left(a_{s}\right)^{\circ} \subset L$
				\item there is no proper parabolic subgroup $P$ of $G^{\circ}$ such that $a \in N_{G} P$ and $C_{G}\left(a_{s}\right)^{\circ} \subset P$.
			\end{enumerate}

			An element $a\in G$ satisfying any of the above conditions is called isolated. An isolated Jordan class is a Jordan class in which every element or equivalently some element is isolated (\cite[3.3]{charactersheafI}).
			By \cite[2.2 (b)]{charactersheafI}, the element $a\in G$ is isolated in $N_G(L(a))$, so, by (\ref{i}), $L(a)=(N_{G}(L(a)))^{\circ}$.\\
			\begin{remark}\label{unipotent jordan class}
				It can be useful to notice that	if $a=a_u$, then $T(a)=(Z(G^{\circ})\cap C_G(a))^{\circ}$. So $J(a)$ is isolated and it is a $G^{\circ}$-orbit traslated by a subgroup of $Z(G^{\circ})$, in particular $J(a)=(Z(G^{\circ})\cap C_G(a))^{\circ}(G^{\circ}\cdot a)$.\\ 
				If $G^{\circ}$ is a simple group, then $T(a)$ is trivial. So the Jordan class of $a$ is the unipotent orbit $G^{\circ}\cdot a$.
			\end{remark}
			
			Now we recall some facts from \cite{charactersheafI} that will be useful in the sequel.
			\begin{remark}\label{proprieta}
				By construction of $T(a)$ and $L(a)$, the following properties hold:
				
				\begin{enumerate}[(a)]
					\item $C_G(a_s)^{\circ}\subset L(a)$,
					\item \label{2} $T(a)=T_{N_G(L(a))}(a)$ (\cite[2.1(d)]{charactersheafI}),
					\item \label{3} $T_{N_G(L(a))}(a)=(Z(L(a))^{\circ}\cap C_{N_G(L(a))}(a))^{\circ}$ (\cite[3.9]{charactersheafI}). 
					\item \label{T}	$T(a)=(Z(L(a))^{\circ}\cap C_{N_G(L(a))}(a))^{\circ}$, by (\ref{2}),(\ref{3}).
				\end{enumerate}
				
			\end{remark}
			Following \cite[1.22]{charactersheafI}, we prove the next proposition.
			\begin{proposition}\label{semisimpleinclosure}
				Let $H$ be an algebraic group, let $D'$ be a connected component of $H$, $h\in D'$, and $x\in \overline{H^{\circ}\cdot h}$. Then $x_s\in H^{\circ}\cdot h_s$.
			\end{proposition}
			\begin{proof}
				Let $C=\{y\in D' \ | \ y_s\in H^{\circ}\cdot h_s \}$ and let  $(H^{\circ}\cdot h)_s=\{f_s \ | \ f\in H^{\circ}\cdot h\}$. We observe that $(H^{\circ}\cdot h)_s$ is the orbit $H^{\circ}\cdot h_s$, so, by \cite[1.4 (e)]{charactersheafI}, is closed. Let, for some $n$, $H\subset GL(n)$ be an embedding of algebraic groups. Let $Y$ be the semisimple class of $GL(n)$ containing $H^{\circ}\cdot h_s$ and let $Y'=\{a\in GL(n) \ | \ a_s \in Y\}$. Then $Y'$ is the set in $GL(n)$ of the matrices that have characteristic polynomial equal to that of a fixed matrix in $Y$, so $Y'$ is closed. We consider the morphism
				$$
				\rho:Y'\longrightarrow Y
				$$
				$$
				a\mapsto a_s.
				$$
				Then $Y'\cap H=\rho^{-1}(H)$ is also closed in $H$ and the restriction of $\rho$ from $Y'\cap H$ to $Y\cap H$ is a morphism of varieties. Hence $C=\rho^{-1}(H^{\circ}\cdot h_s)\cap D'$ is closed in $D'$. So $H^{\circ}\cdot h\subset C$ implies $\overline{H^{\circ}\cdot h}\subset C$.

			\end{proof}

			Let $Q=MU_Q$ be the Levi decomposition of a parabolic subgroup $Q\subset G^{\circ}$. Note that
			$U_Q$ is normalized by $N_GQ$ because $U_Q$ is characteristic in $Q$. Furthermore $U_Q$ acts simply transitively on Levi subgroups of $Q$, hence $U_Q\cap (N_GQ\cap N_GM)={1}$ and the group $(N_GQ\cap N_GM)U_Q$ is isomorphic to the semidirect product of $N_GQ\cap N_GM$ and $U_Q$. By standard arguments, for all $x\in N_GQ\cap N_GM$ the coset $xU_Q$ is $U_Q$-stable.\\
	
			Following the proof of \cite[Proposition 3.15]{charactersheafI}, we show the next proposition.
			\begin{proposition}\label{semisimplepart}
				Let $Q$ be a parabolic subgroup of $G^{\circ}$ with Levi decomposition $Q=MU_Q$. Let $h\in N_GQ\cap N_GM$. Then the semisimple parts of the elements in $hU_Q$ are all $U_Q$-conjugate.
			\end{proposition}
			\begin{proof}
				
				Let $a=hu\in hU_Q$ and let $a=a_sa_u$ be its Jordan decomposition. Then, by the properties of Jordan decomposition, $a_s\in N_GQ$ and by \cite[Proposition 1.6]{dignemichel} it normalizes a Levi subgroup $M'$ of $Q$. Now, $hU_Q$ is $U_Q$-stable, so, since $M'$ is $U_Q$-conjugated to $M$, conjugating $a$ by some element in $U_Q$, we may assume $a_s\in N_GQ\cap N_GM$.
				Let $\pi$ be the projection of the semidirect product $\left(N_{G} Q \cap N_{G} M\right) U_{Q}$ onto $N_{G} Q\cap N_{G} M$ (a homomorphism of algebraic groups). Then $h=\pi(a)$ and $h_{s}=\pi\left(a\right)_s=\pi(a_s) .$ Since $a_{s} \in N_{G} Q \cap N_{G} M,$ we have $\pi\left(a_{s}\right)=a_{s}$ so $a_{s}=h_{s} .$ 
				
			\end{proof}

			Let $D$ be a connected component of $G$.
			Let $Q=MU_Q$ be the Levi decomposition of a parabolic subgroup $Q$ of $G^{\circ}$ such that $N_GQ\cap N_GM\cap D\neq \emptyset$. By \cite[I]{dignemichel} $(N_{G^{\circ}}Q\cap N_{G^{\circ}}M)^{\circ}=N_{G^{\circ}}Q\cap N_{G^{\circ}}M=M$. Let $h\in N_GQ\cap N_GM\cap D$, so $D=G^{\circ}h$. Then 
			
			\begin{equation}\label{levicomponent}
				Mh\subseteq N_GQ\cap N_GM\cap G^{\circ}h=(N_{G}Q\cap N_{G}M\cap G^{\circ})h=Mh.
			\end{equation}
			
			So $N_GQ\cap N_GM\cap D=Mh$ is a connected component of $N_GQ\cap N_GM$.
			\\

			From now on $g$ will always denote an element of $D$ and its Jordan decomposition will be $su$, $L$ will always denote $L(g)$.
			We will denote by $P$ a parabolic subgroup of $G^{\circ}$ with Levi subgroup $L$ and such that $gPg^{-1}=P$, whose existence is ensured by  \cite[2.1 (a)]{charactersheafI}. By \cite[2.2 (b)]{charactersheafI} $g$ is isolated in $N_GL$.
			We will denote by $S$ the (isolated) $N_GL$-Jordan class of $g$. Then $S\subset N_GP$ (\cite[Proof of lemma 3.6]{charactersheafI}), and $S$ is an isolated $N_GP\cap N_GL$-Jordan class.\\
			We will denote the connected component $N_GP\cap N_GL\cap D=Lg$ of $N_GP\cap N_GL$ by $D_L$.
			Let $S^{*}=\{h\in S \ | \ C_G(h_s)^{\circ}\subset L \}$. By \cite[3.9]{charactersheafI}, if $h\in S^{*}$ then $T_G(h)=(Z(L)^{\circ}\cap C_L(h))^{\circ}=T_{N_GL}(h)$.

		\end{subsection}
		\begin{subsection}{Induction of orbits}\label{sectioninduction}
			In this section we recall the induction procedure of $G^{\circ}$-orbits in $G$. The connected case is described in \cite{LS}. The induction of unipotent orbits in a disconnected group is reported in \cite[II, 3]{Spaltestein}.
			
			\begin{definition}
				Let $X$ be  a $G$-variety, $H$ a subgroup of $G$, and $Y$ an $H$-variety. We set
				$$
				G\times^HY=G\times Y/\sim
				$$
				where $(a,y)\sim(a',y')$ if $\exists h\in H$ such that $ah=a'$ and $h^{-1}\cdot y=y'$. We denote the elements of $G\times^HY$ by $[(a,y)]$.
			\end{definition}
			
			There is an action of $G$ on $G\times^HY$ given by
			$$
			b*[(a,y)]=[(ba,y)] \ \ \ \text{for}\ b\in G, \ [(a,y)]\in G\times^HY.
			$$
			Let now $Y$ be a $H$-stable subvariety of $X$.
			Let\\ $Z= \left\{(aH,z)\in G/H\times G\cdot Y \ | \ a^{-1}\cdot z\in Y \right\}\subset G/H\times X$. It is a $G$-variety with action
			$$
			b* (aH,z)=(baH,b\cdot z) \ \ \ \text{for}\ b\in G, \ (aH,z)\in Z.
			$$
			The following lemma is well-known and can be proved by direct verification.
			\begin{lemma}\label{varietadiincidenza}
				Let $X$ be a $G$-variety, let $H\leq G$ and $Y$ an $H$-stable subvariety of $X$. Then 
			
				$$\psi:G\times^HY \longrightarrow Z$$
				$$
				[(a,y)]\mapsto (aH,a\cdot y).
				$$
				is a well defined $G$-equivariant isomorphism of varieties. 
			\end{lemma}

			With notation of Subsection \ref{preliminars}, observe that $\overline{(L\cdot g )}U_P$ is contained in the semidirect product\\ $(N_GP\cap N_GL)U_P$, and $\overline{L\cdot g}$ is $L$-stable, so $\overline{(L\cdot g )}U_P$ is $P$-stable. We consider
			$G^{\circ}\times^P\overline{(L\cdot g)}U_P$.

			\begin{proposition}\label{definitionofphi}
				Let $\phi$ be the following morphism:
			
				\begin{align}\label{defofphi}
					\phi : G^{\circ}\times^P \overline{L\cdot g}U_P\rightarrow G
				\end{align}
				$$
				[(x,y)]\mapsto xyx^{-1}.
				$$
				Then the image of $\phi$	is the closure of a single orbit for the action of $G^{\circ}$. 
				
			\end{proposition}
		
			\begin{proof} 
				Let $X=\left\{(x P,h) \in G^{\circ} / P \times G \ | \ x^{-1} h x \in \overline{L\cdot g} U_{P}\right\}$, and let $\pi: X\rightarrow G$ be the projection to the second factor.
				Then for $\psi:G^{\circ}\times^P \overline{L\cdot g}U_P\xrightarrow{\sim} X$ the $G^{\circ}$-equivariant isomorphism as in Lemma \ref{varietadiincidenza} we have $\pi\circ\psi=\phi$. We show that the image of $\pi$ is the closure of a $G^{\circ}$-orbit. The variety $G^{\circ}/P$ is complete because $P$ is a parabolic group of $G^{\circ}$, and $(\overline{L\cdot g})U_P$ is a closed $P$-stable subset of $N_GP$, so the map $\pi$ is proper.
				Therefore the image of $\pi$ is a closed subvariety of $G$. Furthermore $G^{\circ}\times^P \overline{L\cdot g}U_P$ is a quotient of a product of irreducible varieties, so its image through $\phi$ is irreducible.
				Since the latter is also $G^{\circ}$-stable, it is a union of $G^{\circ}$-orbits. We claim that these orbits are finitely many. Each of these orbits is represented in $\overline{L\cdot g}U_P$ by construction. So we take $x\in \overline{L\cdot g}U_P$ and let $y$ be the unique element in $\overline{L\cdot g}$ such that $x\in yU_P$. By Proposition \ref{semisimplepart} the elements $x_s$ and $y_s$ are $U_P$-conjugated, and $y_s\in \overline{L\cdot g}$. By Proposition \ref{semisimpleinclosure}, $y_s\in L\cdot s$, where $s$ is the semisimple part of $g$. Thus the $G^{\circ}$-orbits of the elements in $ \overline{L\cdot g}U_P$, are in bijection with the unipotent orbits of $C_G(s)$, and the unipotent orbits in $C_G(s)$ are finitely many (\cite[1.15]{charactersheafI}). Since the image of $\pi$ is closed, $G^{\circ}$-stable and irreducible, it is the closure of a single $G^{\circ}$-orbit.
			\end{proof}
			Then the image of $\phi$ from equation (\ref{defofphi}) is the closure of a single orbit, we call this orbit the induced orbit from $D_L$ to $D$ of the orbit $L\cdot g$, and we denote it by  
			$\Ind_{D_L}^{D }(L\cdot g).$
			By construction 
			$$\Ind_{D_L}^D(L\cdot g)=G^{\circ}\cdot (gU_P)^{\text {reg}}.$$\\
			
			Our next goal is to describe the induced orbits in terms of unipotent induced orbits. Recall that we fixed $g=su\in D$. \\
			By \cite[1.12]{charactersheafI} $P\cap G^{s\circ}$ is a parabolic subgroup of $G^{s\circ}$. Hence $P^{s\circ}=P\cap G^{s\circ}$ and it has Levi decomposition $L^{s\circ}U_{P}^{s}$. So we can consider $\Ind_{L^{s\circ}h}^{G^{s\circ}h}(L^{s\circ}\cdot h)$ for the $L^{s\circ}$-orbit of an element $h\in N_{G^s}P^{s\circ}\cap N_{G^s}{L^{s\circ}}$. Note that $G^{s\circ}h$ is the connected component of $G^{s}$ containing $h$, and $L^{s\circ}h$, as observed in (\ref{levicomponent}), is the connected component of $N_{G^s}P^{s\circ}\cap N_{G^s}{L^{s\circ}}$ containing $h$.\\
			
			We will need the following lemma.
			\begin{lemma}\label{representative of induced}
				With notation as above
				$$\Ind_{D_L}^D(L\cdot su)\cap suU_P^{s}\neq \emptyset.$$
			\end{lemma}
			\begin{proof}
				We proceed as in \cite[Lemma 4.4]{CarnovaleEsposito}.\\
				Let $\rho:U_P\times su U_P^s\longrightarrow  U_P\cdot (suU_P^s)$ be the dominant morphism mapping $(v,x)$ to $vxv^{-1}$. There exists an open subset $V$ of $U_P\cdot (suU_P^s)$ such that, $\dim U_P\cdot (suU_P^s)=\dim U_P+ \dim U_P^s -\dim \rho^{-1}(y)$ for all $y\in V$. Using $U_P$-equivariance of $\rho$ we can choose $V$ in such way $V\cap su U_P^s\neq \emptyset$. Let $y\in V\cap su U_P^s$. Then $\rho^{-1}(y)=\{(v,x)\in U_P\times  (suU_P^s) \ | \ v\cdot x =y \}$.
				Let us consider the Levi decomposition $P^{s\circ}=L^{s\circ}U_P^{s}$. Observe that $su\in N_{G^s}P^{s\circ}\cap N_{G^s}L^{s\circ}$, so, by Proposition \ref{semisimplepart}, the elements in $suU_P^s$ have semisimple part conjugated to $s$ by an element of $U_P^s$, so $y_s=s$. If $(v,x)\in \rho^{-1}(y)$, then $vsv^{-1}=vx_sv^{-1}=y_s=s$, so $v\in U_P^s$ and then $\rho^{-1}(y)\cong U_P^s$.
				
				Therefore $\dim (U_P\cdot (suU_P^s))=\dim U_P$ and $\overline{U_P\cdot (suU_P^s)}\subseteq suU_P$ implies the equality.
				By \cite[1]{PMIHES_1965__25__49_0}, $(suU_P)^{reg}$ is dense in $(suU_P)$. Thus $U_P\cdot (suU_P^s)\cap (suU_P)^{reg}\neq \emptyset$.
				
			\end{proof}

			\begin{proposition}\label{unipotentinduction}
				With notation as above
				$$	\Ind_{D_L}^D(L\cdot su)=G^{\circ}\cdot s(\Ind_{L^{s\circ}u}^{G^{s\circ}u}(L^{s\circ}\cdot u)).
				$$
			\end{proposition}
			\begin{proof}
				We proceed as in \cite[Proposition 4.5]{CarnovaleEsposito}. We give the proof for completeness.\\
				Since $\Ind_{D_L}^D(L\cdot su)\cap suU_P^{s}\neq \emptyset$, there holds $(suU_P^{s})^{\text{reg}}\subset (suU_P)^{\text{reg}}$, so
				$$
				s(\Ind_{L^{s\circ}u}^{G^{s\circ}u}(L^{s\circ}\cdot u))=G^{s\circ}\cdot (suU_P^{s})^{\text{reg}}\subset G^{\circ}\cdot (suU_P)^{\text{reg}}=\Ind_{D_L}^D(L\cdot su).
				$$ 
				Thus $\Ind_{D_L}^D(L\cdot su)=G^{\circ}\cdot s(
				\Ind_{L^{s\circ}u}^{G^{s\circ}u}(L^{s\circ}\cdot u))$.
				
			\end{proof}

		\end{subsection}
		\begin{subsection}{Closure and regular closure}
			In \cite{charactersheafI} G. Lusztig described the closure of a Jordan class. Building on this, we describe the closure and the regular closure of a Jordan class in terms of induced orbits.\\
			We retain notation of Section \ref{sectioninduction}. We observe that, since the set $\overline SU_P$ is contained in the semidirect product $(N_GP\cap N_GL)U_P$, and $\overline S$ is $L$-stable, then $\overline SU_P$ is $P$-stable.\\This allows us to consider the variety	
	
			$$
			\widetilde{X}=\left\{(x P,h) \in G^{\circ} / P \times G \ | \ x^{-1} h x \in \overline{S} U_{P}\right\}.
			$$ 
			By \cite[Lemma 3.14]{charactersheafI}  the image of $\widetilde{X}$ through the projection $\pi$ on the second factor of $\widetilde{X}$ is $\overline{J(g)}$, so
			$$
			\pi(\widetilde{X})=\bigcup_{x\in G^{\circ}}x\overline{S}U_Px^{-1}=G^{\circ}\cdot\overline{S}U_P=\overline{J(g)},
			$$
			and also it is an union of Jordan classes.
			Hence by the identification in Lemma \ref{varietadiincidenza}, the image of the map 
			
			\begin{equation}\label{phi}
				\tilde{\phi}:G^{\circ}\times^P\overline{S}U_P \longrightarrow G
			\end{equation}
			$$
			[(x,y)]\mapsto xyx^{-1}
			$$
			is $\overline{J(g)}$.\\

			Now, we look at the structure of isolated Jordan classes.

			By \cite[3.3 (a)]{charactersheafI},
			any isolated Jordan class of a reductive group $H$ is a single orbit for the action of $Z(H^{\circ})^{\circ} \times H^{\circ}$ on $H$:
			\begin{equation}\label{centeraction}
				(Z(H^{\circ})^{\circ}\times H^{\circ})\times H\rightarrow H
			\end{equation}
			$$
			((z,x),y)\mapsto xzyx^{-1} .
			$$
			
			We apply it to the case  $H=(N_GP\cap N_GL)$, so the isolated Jordan class $S$ in $N_GP\cap N_GL$ considered above, is an orbit for $Z(L)^{\circ}\times L$. 
			Then 
			
			$$
			S=Z(L)^{\circ}(L\cdot g).
			$$

			We need the following lemma.
			\begin{lemma}\label{closureisolated}
				With notation as above
				\begin{enumerate}[(a)]
					\item $S=Z(L)^{\circ}([L,L]\cdot g)$;
					\item $\overline S= Z(L)^{\circ}(\overline{L\cdot g})=Z(L)^{\circ}(\overline{[L,L]\cdot g})$;
					\item $S=T(g)(L\cdot g)$;
					\item $\overline S=T(g)(\overline{L\cdot g})$.
					
				\end{enumerate}
			\end{lemma}
			\begin{proof}\phantom{aaa}
				\begin{enumerate}[(a)]
					\item \label{primo punto proposizione isolati} Obviously $Z(L)^{\circ}([L,L]\cdot g)\subset Z(L)^{\circ}(L\cdot g)=S$. Conversely, the subgroup $Z(L)^{\circ}$ is a characteristic subgroup of $L$, so, since $g\in N_GL$, $g\in N_G(Z(L)^{\circ})$. In addition $L$ is reductive, so $L=Z(L)^{\circ}[L,L]$. Hence, if $x\in Z(L)^{\circ}(L\cdot g)$, there exist $z_1,z_2\in Z(L)^{\circ}$ and $l\in [L,L]$ such that $x=z_1z_2lgl^{-1}z_2^{-1}$. Since $z_2^{-1}\in Z(L)^{\circ}$ and $lgl^{-1}\in N_G(Z(L)^{\circ})$, 
					$$
					x= z_3lgl^{-1},
					$$
					for some $z_3\in Z(L)^{\circ}$. Therefore $Z(L)^{\circ}(L\cdot g)\subset Z(L)^{\circ}([L,L]\cdot g)$ and $S=Z(L)^{\circ}(L\cdot g)=Z(L)^{\circ}([L,L]\cdot g)$.
					
					\item \label{point b}
					The derived subgroup $[L,L]$ is characteristic in $L$. So $g\in N_G([L,L])$, then $[L,L]\cdot g\subset [L,L]g$.\\
					The morphism $m: Z(L)^{\circ}\times [L,L] g\longrightarrow Lg$ given by the multiplication, is finite (\cite[Prop 2.3.2, Theorem 2.4.9]{Springer}) hence closed. Thus $Z(L)^{\circ}(\overline{[L,L]\cdot g})=m(Z(L)^{\circ}\times \overline{[L,L]\cdot g})$ is closed in $Lg$, since $Z(L)^{\circ}\times \overline{[L,L]\cdot g}\subset Z(L)^{\circ}\times [L,L] g$ is closed.\\ 
					
					Clearly $\overline{[L,L]\cdot g}\subset \overline{ Z(L)^{\circ}([L,L]\cdot g)}=\overline S$, so by $Z(L)^{\circ}$-stability of $\overline S$ we obtain
					$$S=Z(L)^{\circ}([L,L]\cdot g)\subset Z(L)^{\circ}(\overline{[L,L]\cdot g})=\overline{ Z(L)^{\circ}([L,L]\cdot g)}=\overline S.$$
					Since $Z(L)^{\circ}(\overline{[L,L]\cdot g})$ is closed, thus $Z(L)^{\circ}(\overline{[L,L]\cdot g})=\overline S$.
					
					By point (\ref{primo punto proposizione isolati}), $\overline S=\overline{ Z(L)^{\circ}([L,L]\cdot g)}$.
					Furthermore, from the $Z(L)^{\circ}$-stability of $\overline S$, we get
					$$
					Z(L)^{\circ}(\overline{[L,L]\cdot g})\subset Z(L)^{\circ}(\overline{L\cdot g})\subset \overline S=Z(L)^{\circ}(\overline{[L,L]\cdot g}).
					$$
					So the equality holds.
					\item \label{point c}	By \cite[1.21(d)]{charactersheafI}, since $S\in D_L$, it is an orbit for the action (\ref{centeraction}) restricted to $(Z(L)^{\circ}\cap C_{N_GL}(g))\times L$. By Remark \ref{proprieta} (\ref{T}) there holds $Z(L)^{\circ}\cap C_{N_GL}(g)=T(g)$,
					so
					$$
					S=T(g)(L\cdot g).
					$$
					
					\item By point (\ref{point b}) $Z(L)^{\circ}\overline{L\cdot g}=\overline S$. Hence, since $T(g)\subset Z(L)^{\circ}$,
					$$
					T(g)\overline{L\cdot g}\subset Z(L)^{\circ}\overline{L\cdot g}=\overline S,
					$$		
					
					By \cite[1.21 (b)]{charactersheafI}, for a $L$-stable subset $X$ of $D_L$, we have that if $T(g)X\subset X$ then $Z(L)^{\circ}X\subset X$.\\ Let $X=T(g)\overline{L\cdot g}$. So obviously $T(g)X\subset X$, then $Z(L)^{\circ}X\subset X$, that is $Z(L)^{\circ}T(g)\overline{L\cdot g}=Z(L)^{\circ}\overline{L\cdot g}=\overline{S}\subset T(g)\overline{L\cdot g}$. 
					Thus $\overline S=T(g)\overline{L\cdot g}$.
				\end{enumerate}
			\end{proof}
			\begin{proposition}\label{closurejordanclass}
				With notation as above
				$$
				\overline{J(g)}=\bigcup_{z\in T(g)}\overline{\Ind_{D_L}^D(L\cdot zg)}.
				$$
			\end{proposition}
			\begin{proof}
				
				Let $z\in T(g)$ and let $\phi_z$ be the map 
				$$
				\phi_z : 	G^{\circ}\times^P\overline{L\cdot zg}U_P\longrightarrow D
				$$
				$$
				[(a,x)]\mapsto axa^{-1}.
				$$
				By Proposition \ref{definitionofphi}, the image of $\phi_z$ is $\overline{\Ind_{D_L}^D(L\cdot zg)}$. By Lemma \ref{closureisolated} (d)
				$$
				G^{\circ}\times^P\overline{S}U_P=G^{\circ}\times^P\bigcup_{z\in T(g)}\overline{L\cdot zg}U_P=\bigcup_{z\in T(g)}G^{\circ}\times^P\overline{L\cdot zg}U_P
				$$
				Then $\im\widetilde{\phi}=\bigcup_{z\in T(g)}\im\phi_z$, where $\widetilde{\phi}$ is the map defined in (\ref{phi}), so
				$\overline{J(g)}=\bigcup_{z\in T(g)}\overline{\Ind_{D_L}^D(L\cdot zg)}.$
				
			\end{proof}
			We study now the regular closure of $J(g)$, that is $\overline{J(g)}^{\reg}$.\\

			Let $d\in \N$ be such that $J(g)\subset G_{(d)}$. Then $\overline{J(g)}\subset \overline{G_{(d)}}\subset\bigcup_{k\leq d}G_{(k)}$, so $J(g)\subset  \overline{J(g)}^{\text{reg}}= G_{(d)}\cap \overline{J(g)}$. \\
			Let $w$ be a representative of $\Ind_{L^{s\circ}u}^{G^{s\circ}u}(L^{s\circ}\cdot u)$.
			By Proposition \ref{unipotentinduction},\\ $
			\Ind_{D_L}^D(L\cdot su)=G^{\circ}\cdot sw$. Hence
			$$	\codim_D\Ind_{D_L}^D(L\cdot su)=\codim_{D}G^{\circ}\cdot sw=\dim (G^{s\circ}\cap G^{w\circ})=$$
			\begin{align}\label{codim}
				=	\dim G^{s\circ}-\dim G^{s\circ}\cdot w=\codim_{G^{s\circ}u} G^{s\circ}\cdot w.
			\end{align}

			\begin{proposition}\label{codimensione}
				With notation as above
				$$
				\codim_{D_L}(L\cdot g)=\codim_{D}\Ind_{D_L}^D(L\cdot g)
				$$ 
			\end{proposition}
			\begin{proof}
				There holds 
				\begin{flalign*}
					&	\codim_D \Ind_{D_L}^D(L\cdot su)=\codim_{G^{s\circ}u} \Ind_{L^{s\circ} u}^{G^{s\circ}u}(L^{s\circ}\cdot u)=\\
					&=	\codim_{L^{s\circ}u}L^{s\circ}\cdot u=
					\dim L^{s\circ}u-\dim L^{s\circ}\cdot u=\\
					&		=\dim L^{s\circ}-\dim L^{s\circ}\cdot u=\dim L^{s\circ}-\dim L^{s\circ}+\dim (L^{s\circ}\cap L^{u\circ})=\\
					&	=		\dim (L^{s\circ}\cap L^{u\circ})=\codim_{D_L}L\cdot su.
				\end{flalign*}
				where the first equality is (\ref{codim}), the second is \cite[Proposition 3.2]{Spaltestein}, and the others follow from properties of the Jordan decomposition and the equality $\dim D_L=\dim L$.
			\end{proof}
			\begin{remark}\label{Lzso=Gso}
				By \cite[3.9]{charactersheafI}, $G^{s\circ}\subseteq L$. Then $L^{zs\circ}=G^{s\circ}$, for any $z\in Z(L)^{\circ}$. Indeed $G^{s\circ}\subset L\cap G^{zs}\subset L^{zs}$, and then if $x\in L^{zs\circ}\subset L\cap G^{zs}$, then $x$ commutes with $zs$ and with $z$, so $x$ also commutes with $s$, so $x\in G^{s}$, thus $L^{zs\circ}\subset G^{s\circ}$. 
			\end{remark}
			
			\begin{proposition}\label{reg}
				With notation as above
				\begin{enumerate}[(a)]
					\item \label{reg a}
					$	\overline{J(g)}^{\text{reg}}=\bigcup_{z\in T(g)}\Ind_{D_L}^D(L\cdot zg),
					$
					\item \label{reg b} $\overline{J(g)}^{\text{reg}}=\bigcup_{z\in T(g)}G^{\circ}\cdot zs\Ind_{G^{s\circ}u}^{G^{zs\circ}u}(G^{s\circ}\cdot u).$
				\end{enumerate}
			\end{proposition}
			\begin{proof}
				\phantom{aaa}
				\begin{enumerate}[(a)]
					\item \label{a} Let  $\mathcal{O}$ be a $G^{\circ}$-orbit in $\overline{J(g)}$ of maximal dimension. By Proposition \ref{closurejordanclass}, $\exists z\in T(g)$ such that $\mathcal{O}\subset \overline{\Ind_{D_L}^D(L\cdot zg)}$, so, by maximality, $\mathcal{O}=\Ind_{D_L}^D(L\cdot zg)$. So there exists an induced orbit in $\overline{J(g)}$ that has dimension $d$, where $d\in \N$ is such that $J(g)\subset G_{(d)}$. We claim that all the induced $G^{\circ}$-orbits from $L$-orbits in $S$ have the same dimension.  Since $S$ is a Jordan class, the dimension of its $L$-orbits concide. The claim follows from Proposition \ref{codimensione}.
					\item Let $z\in T(g)$. Then $z$ is semisimple and lies in $G^{s}\cap G^u$, so $(zg)_s=zs$ and $(zg)_u=u$. Therefore, by Proposition \ref{unipotentinduction} and point (\ref{a})
					$$\overline{J(g)}^{\text{reg}}=\bigcup_{z\in T(g)}\Ind_{D_L}^D(L\cdot zg)=\bigcup_{z\in T(g)}G^{\circ}\cdot zs\Ind_{L^{zs\circ}u}^{G^{zs\circ}u}(L^{zs\circ}\cdot u)=$$
					$$=\bigcup_{z\in T(g)}G^{\circ}\cdot zs\Ind_{G^{s\circ}u}^{G^{zs\circ}u}(G^{s\circ}\cdot u),$$		
					where the last equality follows from Remark \ref{Lzso=Gso}.
				\end{enumerate}
			\end{proof}

		\end{subsection}
		\begin{subsection}{The poset of Jordan classes}\label{poset}
			Let $J_1,J_2$ be Jordan classes. In \cite[7.2 (c)]{charactersheafII} G. Lusztig described when $J_1\subset \overline{J_2}$. In this section we describe when $J_1\subset \overline{J_2}^{\reg}$ in terms of induced orbits.

			We define a partial order on the set of Jordan classes:
			$$J_1\leq J_2 \text{ if and only if } J_1\subseteq \overline{J_2}^{\text{reg}}.$$
			
			Let $g_i\in J_i$. Set $L_i=L(g_i)$ and let $P_i$ be a parabolic subgroup of $G^{\circ}$ with Levi $L_i$ and such that $g_i\in N_GP_i$. Let $S_i$ be the isolated Jordan class in $N_GP_i\cap N_GL_i$ of $g_i$, for $i=1,2$.\\
			Let $g_1=tv$ and $g_2=su$ be respectively the Jordan decomposition of $g_1$ and $g_2$.
			By construction of $S_i$ we have $tv\in S_1^{*}$ and $su\in S_2^{*}$, where, as we saw in the first Section, $S_i^{*}=\{h\in S_i \ | \ C_G(h_s)^{\circ}\subset L_i\}$. We recall that, by \cite[3.9]{charactersheafI}, if $h_i\in S_i^{*}$ then $T_G(h_i)=T_{N_GL}(h_i)$ and $L(h_i)=L_i$. Therefore

			\begin{itemize}
				\item $T_G(tv)=T_{N_GL_1}(tv)$ and $T_G(su)=T_{N_GL_2}(su)$,
				\item $L(tv)=L_1$ and $L(su)=L_2$.
			\end{itemize}
	
			\begin{proposition}
				With the notation above if $J_1\leq J_2$,
				then $\exists y\in G^{\circ}$ such that
				\begin{enumerate}[(a)]
					\item \label{classification c}$v\in y\cdot (\Ind_{G^{s\circ}u}^{G^{zs\circ}u}(G^{s\circ}\cdot u))$;
					\item \label{classification a} $T(tv)t\subset y\cdot (T(su)s)$;
					\item \label{classification b} $y\cdot (G^{s\circ})\subset G^{t\circ}$.
					
				\end{enumerate}			
			\end{proposition}
			
			\begin{proof}
				Since $J_1\leq J_2$, by Proposition \ref{reg}, $$J_1\subset   \bigcup_{z\in  T(su)}G^{\circ}\cdot (zs\Ind_{G^{s\circ}u}^{G^{zs\circ}u}(G^{s\circ}\cdot u)).$$ Then $\exists z\in T(su)$ and $\exists y\in G^{\circ}$ such that $tv\in y\cdot (zs\Ind_{G^{s\circ}u}^{G^{zs\circ}u}(G^{s\circ}\cdot u))$. By Lemma \ref{representative of induced}, we can choose an element in $\Ind_{G^{s\circ}u}^{G^{zs\circ}u}(G^{s\circ}\cdot u)$ of the form $uu'\in uU_{P_2}^{zs\reg}$. Therefore $tv=y\cdot zsuu'$. 
				We observe that, since $z\in T(su)$, it is semisimple and commutes with $s$, so $zs$ is semisimple. And also $zsuu'=uu'zs$ with $uu'$ unipotent element. So $t=(tv)_s=y\cdot (zsuu')_s=y\cdot zs$ and $v=y\cdot(tv)_u=y\cdot(zsuu')_u=y\cdot uu'$, giving (\ref{classification c}).
				\\
				
				Now we show that $T(zsuu')\subset T(su)$.\\
				We observe that, since $L_2^{zs\circ}$ is a Levi subgroup of $G^{zs\circ}$, then $Z(G^{zs\circ})\subset Z(L_2^{zs\circ})$, so 
				\begin{align}\label{formula toro}
					T(zsuu')=(Z(G^{zs\circ})\cap C_G(uu'))^{\circ} \text { and }
				\end{align}
				$$
				(Z(G^{zs\circ})\cap C_G(uu'))^{\circ}	\subset (Z(L_2^{zs\circ})\cap C_G(uu'))^{\circ}=(Z(L_2^{zs\circ})\cap C_{L_2}(uu'))^{\circ}.
				$$
				Let $l\in C_{L_2}(uu')$. Then $uu'=luu'l^{-1}=lul^{-1}lu'l^{-1}$, where $lul^{-1}\in N_GL_2$ and $lu'l^{-1}\in U_{P_2}$. By the uniqueness of the decomposition in the semidirect product $N_GL_2U_{P_2}$, $lul^{-1}=u$. Thus 
				$$
				C_{L_2}(uu')\subset C_{L_2}(u)\subset C_G(u).
				$$
				Therefore, combining (\ref{formula toro}) with Remark \ref{Lzso=Gso}, we get
				$$
				T(zsuu')\subset(Z(L_2^{zs\circ})\cap C_{L_2}(uu'))^{\circ}\subset (Z(G^{s\circ})\cap C_{G}(u))^{\circ}=T(su).
				$$ 
				Furthermore $z\in T(su)$, so $T(zsuu')zs\subset T(su)s$. Thus \\$T(tv)t=y\cdot (T(zsuu')zs)\subset y\cdot (T(su)s)$, proving (\ref{classification a}).\\
				
				By Remark \ref{Lzso=Gso}, $y\cdot G^{s\circ}=y\cdot L^{zs\circ}\subset y\cdot G^{zs\circ}=G^{t\circ}$, then (\ref{classification b}).
				
			\end{proof}
			
		\end{subsection}
	\end{section}
	
	\begin{section}{Lusztig strata}
		In this section we show that the strata of a disconnected reductive algebraic group defined by Lusztig in \cite{lusztig2020strata} are locally closed using the strategy from \cite{carno}.
		We recall the definition of the strata.\\
		
		Let $W$ be the Weyl group of $G^{\circ}$, viewed as the set of $G^{\circ}$-orbits on $\mathcal{B}\times \mathcal{B}$, where $\mathcal{B}$ is the flag manifold of $G^{\circ}$. 
		The group $G$ acts by inner automorphisms on $G^{\circ}$ and this induces an action of $G / G^{\circ}$ on $W$. In particular, $D$, viewed as an element of $G / G^{\circ}$, defines an automorphism $[D]: W \rightarrow W$ whose fixed point set is denoted by $W^{D}$. It is the Weyl group of some root system (\cite[Appendix]{Lusztig2003HeckeAW}). We set $D_{un}=\{a\in D \ | \ a=a_u\}$.
		
		Let $E$ be the map from $D$ to the set of isomorphism classes of irreducible representations of the group $W^D$ defined in \cite{lusztig2020strata} so
		$$
		E: D\longrightarrow \Irr(W^D)
		$$
		$$
		a\mapsto E(a)
		$$
		where $E(a)$ is constructed according to the following rules
		\begin{itemize}
			\item if $a=a_u$, then $E(a)$ is the unique irreducible representation of $W^D$ such that $\pi_{!}(\overline{\mathbb{Q}_l}[\dim D]_E)|D_{un}$ is (up to shift) the intersection cohomology complex of $\overline{G^{\circ}\cdot a}$ with coefficients in $\overline{\mathbb{Q}_l}$, where $\pi$ is defined as in \cite[1.2]{lusztig2020strata}. 
			\item if $a_s$ is central in $G$, let $D_{a_u}$ be the connected component of $G$ containing $a_u$. We observe that $W^{D_{a_u}}=W^D$, because $D_{a_u}=G^{\circ}a_u$ and $D=G^{\circ}a_sa_u=a_sG^{\circ}a_u=a_s D_{a_u}$ with $a_s$ central. Then $E(a)\coloneqq E(a_u)$.
			\item if $a_s\notin Z(G)$, then $a_s$ is central in $C_G(a_s)$, and we let $D'$ be the connected component of $C_G(a_s)$ containing $a$. We denote by $E'(a)\in \Irr(W(C_G(a_s)^{\circ})^{D'})$ the image of $a$ through the map $E$ referred to the group $C_G(a_s)$. Then we set $E(a)=\textbf{j}_{W(C_G(a_s)^{\circ})^{D'}}^{W(G^{\circ})^D}E'(a)$, where $\textbf{j}$ is the truncated induction as defined in \cite{LS}.
		\end{itemize}

		The fibers of $E$ are called the Lusztig strata of $D$. By \cite[1.16 (e)]{lusztig2020strata} Lusztig strata are union of $G^{\circ}$-orbits of the same dimension, so if $X$ is a Lusztig stratum contained in $D$, then there exists $d\in \mathbb{N}$ such that $X\subset G_{(d)}\cap D$.\\
		
		We recall a basic property of $E$.
		\begin{proposition}\label{finite union of jordan classes}
			If $J$ is a Jordan class, and $a=a_sa_u,h=h_sh_u\in J$, then $E(a)=E(h)$, so strata are unions of Jordan classes.
		\end{proposition}
		\begin{proof}
			By Remark \ref{centsemisuguali}, up to conjugation $C_G(a_s)^{\circ}=C_G(h_s)^{\circ}$. Furthermore $h=za$ where $z$ is in $C_G(a_s)^{\circ}$, so $a$ and $h$ are in the same connected component of $C_G(a_s)$. Hence the connected component $D^{'}_h$ of $C_G(h_s)$ containing $h$ and the connected component $D^{'}_a$ of $C_G(a_s)$ containing $a$ concide, as well as the unipotent parts of $a$ and $h$. Thus $E(h)=E(a)$.
		\end{proof}

		By Proposition \ref{reg} (b), if $x\in \overline{J(su)}^{\text{reg}}$, then $\exists z\in T(su)$ and $v\in \Ind_{G^{s\circ}u}^{G^{zs\circ}u}(G^{s\circ}\cdot u)$ such that, up to conjugation, $x=zsv$.\\
		To simplify notation, we set:
		\begin{itemize}
			\item $D'\coloneqq sG^{s\circ}u$ the connected component of $G^{s}$ containing $su$,
			\item $D'_u\coloneqq G^{s\circ}u$ the connected component of $G^{s}$ containing $u$,
			\item $\widetilde{D}$ the connected component of $G^{zs}$ containing $x$
			\item $\widetilde{D_u}\coloneqq G^{zs\circ}u$ the connected component of $G^{zs}$ containing $u$,
			\item   $\widetilde{D'_u}\coloneqq L^{zs\circ}u$ the connected component of $N_{G^{zs}}P^{zs\circ}\cap N_{G^{zs}}L^{zs\circ}$ containing $u$.
			
		\end{itemize}

		Since $x=zsv$ and $v\in \Ind_{D'_u}^{\widetilde{D_u}}(G^{s\circ}\cdot u)$ then $v\in \widetilde{D_u}$, so $W(G^{zs\circ})^{\widetilde{D}}=W(G^{zs\circ})^{\widetilde{D_u}}$.
		Observe  also that, by Remark \ref{Lzso=Gso}, $\widetilde{D'_u}= D'_u$. Then $W(G^{s\circ})^{D'}=W(G^{s\circ})^{\widetilde{D'_u}}$.

		The statement of the following result was suggested by G. Lusztig in \cite[1.12 (b)]{lusztig2020strata}.
		\begin{proposition}\label{strata}
			A stratum $X$ is locally closed.
		
		\end{proposition}
		\begin{proof}
		 
			We first prove  that if $J(su)\subset X$ then $\overline{J(su)}^{\text{reg}}\subset X$. Let $x\in\overline{J(su)}^{\text{reg}}$. We show that $E(x)=E(su)$. With notation as before 
		
			$\exists z\in T(su)$ and $v\in \Ind_{ \widetilde{D'_u}}^{\widetilde{D_u}}(G^{s\circ}\cdot u)$ such that $x=zsv$.

			By \cite[1.9 (iii)]{lusztig2020strata} $E(u)$ is good, so we can apply $\textbf{j}$-induction from $W(G^{s\circ})^{\widetilde{D'_u}}$ to $W(G^{zs\circ})^{\widetilde{D_u}}$ obtaining an irreducible representation. By \cite[Corollary 2.10]{Sorlin}, $$E(v)=\textbf{j}_{W(G^{s\circ})^{\widetilde{D'_u}}}^{W(G^{zs\circ})^{\widetilde{D_u}}}E(u).$$
			
			By definition of $E$, $$E(x)=\textbf{j}_{W(G^{zs\circ})^{\widetilde{D}}}^{W(G^{\circ})^D}E(v).$$
			
			So putting together these relations and using transitivity of the truncated induction, we have
			$$
			E(x)=\textbf{j}_{W(G^{zs\circ})^{\widetilde{D}}}^{W(G^{\circ})^D}E(v)=\textbf{j}_{W(G^{zs\circ})^{\widetilde{D}}}^{W(G^{\circ})^D}\textbf{j}_{W(G^{s\circ})^{\widetilde{D'_u}}}^{W(G^{zs\circ})^{\widetilde{D_u}}}E(u)=\textbf{j}_{W(G^{s\circ})^{\widetilde{D'_u}}}^{W(G^{\circ})^D}E(u).
			$$
			
			By definition of $E$, there holds $E(su)=\textbf{j}_{W(G^{s\circ})^{D'}}^{W(G^{\circ})^D}E(u)$, so $E(x)=E(su)$ and $x\in X$.
			
			By Proposition \ref{finite union of jordan classes} $X$ is a finite union of Jordan classes (\cite{lusztig2020strata}), therefore $X$ is a finite union of regular closures of Jordan classes. Let $I$ be the set parametrizing them and let $X\subset G_{(d)}$. Then
			$$
			X=\bigcup_{i\in I} \overline{J(g_i)}^{\text{reg}}=\overline{\bigcup_{i\in I}J(g_i)}\cap G_{(d)}=\left(\overline{\bigcup_{i\in I}J(g_i)}\right)\cap \left(\bigcup_{i> d-1}G_{(i)}\right),
			$$
			is the intersection of a closed subset and an open subset of $D$, so $X$ is locally closed.
		
		\end{proof}
	\end{section}

	\begin{section}{Examples}
		Let $\mathbb{K}$ be an algebraic closed field of characteristic $2$.
		In this section we will analize the Jordan classes in the semidirect product of $SL(n)$ with the group generated by a non-inner automorphism of $SL(n)$.

		Let $\tau$ be the involution of $SL(n)$, defined as follows
		$$
		\tau :SL(n)\longrightarrow SL(n)
		$$
		$$
		X\mapsto J\ ^tX^{-1}J,
		$$
		with $J=\begin{pmatrix}
			&&1 \\ & \iddots & \\ 1&&\end{pmatrix}$, and $^tX^{-1}$ the transposed of the inverse of $X$.\\
		
		Let $T=\left\{\begin{pmatrix}*&&\\&\ddots&\\&&*\end{pmatrix}\right\}$ be the maximal torus in $SL(n)$ consisting of diagonal matrices, and $B=\left\{\begin{pmatrix}*&\cdots&*\\&\ddots&\vdots\\&&*\end{pmatrix}\right\}\subset SL(n)$ the Borel subgroup of upper triangular matrices. Let $U$ be the subgroup of unipotent elements in $B$. \\
		Then $\tau$ preserves $B,T$ and $U$. So $\tau$ induces an automorphism of the Dynkin diagram of $SL(n)$.
		We consider the non connected algebraic group $G=SL(n)\rtimes \langle\tau\rangle$, so $G^{\circ}=SL(n)$.\\
		Since $|\tau|=\charact\mathbb{K}$, the element $\tau$ is a unipotent element of $G$. Furthermore the automorphism $\tau$ is such that $\tau T=T$ and $\tau B=B$, so it is a quasi semisimple element (\cite[1.4]{charactersheafI}). Then we can apply the result in \cite[8.2 pp 52]{steinberg}, and we obtain that $G^{\tau}$ is connected and reductive. In \cite[Theorem 1.2 (iii)]{outerunipotent} there is the description of the reductive part of the centralizer of a unipotent element. Using this result and the fact that $G^{\tau}$ is reductive, we can conclude that  $C_{G^{\circ}}(\tau)=(G^{\circ})^{\tau}=Sp(2\lfloor\frac{n}{2}\rfloor)$.

		We analize the Jordan classes of elements of $G^{\circ}\tau$.

		\begin{remark}\label{semisimpletau}
			Let $H$ be an algebraic group in a field of characteristic $p$, let $\sigma$ be an automorphism of $H^{\circ}$. Then, by \cite[pp.51]{steinberg}, $\sigma$ is semisimple if and only if there exists an integer $l$ coprime with $p$ such that $\sigma^{l}$ is a semisimple inner automorphism of $H^{\circ}$, i.e., the action of $\sigma^{l}$ is conjugation by a semisimple element of $H^{\circ}$.\\
			Therefore in $G^{\circ}\tau$ there are no semisimple elements.
		\end{remark}
		
		\begin{remark}
			
			By \cite[Lemma 5]{twistedspringer}, every element in $G^{\circ}\tau$ is $G^{\circ}$-conjugated to an element $tu\tau$ with $t\in T^{\tau}$ and $u\in U$.

			Let $B'=T^{\tau}\ltimes U$ and let $h\in B'\tau$. By Remark \ref{semisimpletau}, $h_s\in B'$ and $h_u=h_u'\tau\in B'\tau$.
			Observe that $B'$ is a solvable group. Then (\cite[19.3]{humphreys}), all maximal tori in $B'$ are conjugated by elements in $U$. Hence, there exists $v\in U$ such that $vh_sv^{-1}=z\in T^{\tau}$. 
			Thus $vhv^{-1}=vh_sv^{-1}vh_u'\tau( v^{-1})\tau=zvh_u'\tau(v^{-1})\tau$. So, up to conjugation, every element  $x\in G^{\circ}\tau$ is of the form $tu\tau$ with $t\in T^{\tau}$ and $tu=ut$.
		\end{remark}	
		
		Let $x=tu\tau \in B'\tau$, with $t\in T^{\tau}$ and $u\in U$, with $tu=ut$, and let $k=\lfloor\frac{n}{2}\rfloor$.\\
		
		An element in $T^{\tau}$ is of the form: 
		\begin{itemize}
			\item  if $n$ is even, then \begin{equation}\label{form of t}
				t=\diag(a_1,\dots,a_k,a_k^{-1},\dots,a_1^{-1}), \text{ for } a_1,\dots,a_k\in \mathbb{K}^{\star},
			\end{equation}

			\item	 if $n$ is odd, then 	\begin{equation}\label{form of t 2}
				t=\diag(a_1,\dots,a_k,1,a_k^{-1},\dots,a_1^{-1}), \text{ for } a_1,\dots,a_k\in \mathbb{K}^{\star}.
			\end{equation}

		\end{itemize}
		
		Observe that $t$ is regular if and only if $a_i\neq a_j$, $a_i\neq a_j^{-1}$ and $a_i\neq 1 $ for all $i\neq j\in \{1,\dots,k\}$.
		So, if $t$ is regular, the condition $tu=ut$ implies $u=\Id$, and $x=t\tau$.\\

		We describe the Jordan class $J(t\tau)$ for $t$ a regular semisimple element of $G$ contained in $T^{\tau}$.\\
		It is well known that in this case $C_G(t)^{\circ}=T$. So

		Therefore $T(t\tau)=(T\cap C_G(\tau))^{\circ}=(T^{\tau})^{\circ}=T^{\tau}$. 
		So
		$z\in T(t\tau)^{\reg}$ is of the form
		
		\begin{itemize}
			\item	if $n=2k+1$, 
			$$
			z=\diag(z_1,\dots,z_k,1,z_k^{-1},\dots,z_1^{-1}) \ \text{where } z_i\neq0, \  z_i\neq z_j^{\pm1} \ \forall i,j\in\{1,\dots,k\},
			$$
			\item if $n=2k$,
			$$
			z=\diag(z_1,\dots,z_k,z_k^{-1},\dots,z_1^{-1}) \ \text{where } z_i\neq0, \  z_i\neq z_j^{\pm1} \ \forall i,j\in\{1,\dots,k\}.
			$$
		\end{itemize}
		Then $J(t\tau)=SL(n)\cdot (T(t\tau)^{\reg}\tau)$.\\

		Now we consider $x=tu\tau\in B'\tau$ with $tu=ut$, $t$ not regular, and $t$ as in (\ref{form of t}) or (\ref{form of t 2}).\\
		Then there exist coefficients of $t$ that coincide.
		We can assume that the coinciding elements are among $a_1,\dots a_k$, because the other cases are conjugated to this one by a monomial matrix in $(G^{\circ})^{\tau}$, so they produce the same Jordan class.\\
		So suppose that there exist $a_1,\dots,a_{r-1}\in \mathbb{K}^{\star}$ such that $a_i\neq a_j^{\pm 1}$ for all $i,j=1,\dots,r$, and, up to conjugation, 
		$$t=
		\begin{pmatrix}
			a_1\Id_{h_1}&&&&&&\\&\ddots&&&&&\\&&a_{r-1}\Id_{h_{r-1}}&&&&\\&&&\Id_{h_{r}}&&&&\\&&&&a_{r-1}^{-1}\Id_{h_{r-1}}&&\\&&&&&\ddots&\\&&&&&&a_1^{-1}\Id_{h_1}
		\end{pmatrix},
		$$
		with $n$ even, $h_i\in \mathbb{N}$ and $h_r\geq0$ even, and 
		$$t=
		\begin{pmatrix}
			a_1\Id_{h_1}&&&&&&\\&\ddots&&&&&\\&&a_{r-1}\Id_{h_{r-1}}&&&&\\&&&\Id_{h_r}&&&\\&&&&a_{r-1}^{-1}\Id_{h_{r-1}}&&\\&&&&&\ddots&\\&&&&&&a_1^{-1}\Id_{h_1}
		\end{pmatrix},
		$$
		with $n$ odd, $h_i\in \mathbb{N}$ and $h_r\geq1$ odd, 
		where $\Id_{h_i}$ is the $h_i\times h_i$-identity matrix.\\
		Let $z\in T(x)$. By \cite[Theorem 1.8]{dignemichel}, the subgroup $T\cap C_G(t)^{\circ}$ is a maximal torus of $C_G(t)^{\circ}$. This implies $Z(C_G(t)^{\circ})^{\circ}\subseteq T$. So $z\in T\cap C_G(u\tau)$. Since
		$zu\tau(z^{-1})=u\in U
		$ and $z\in T$, looking at the projection from $TU$ to $T$, we see that $z\in T^{\tau}$.

		Furthermore, from the condition $z\in Z(C_G(t)^{\circ})^{\circ}$ there exist $z_1,\dots,z_{r-1}\in \mathbb{K}^*$ such that 
		$$
		z=\diag(z_1\Id_{h_1},\dots, z_{r-1}\Id_{h_{r-1}},\Id_{h_r},z_{r-1}^{-1}\Id_{h_{r-1}},\dots, z_1^{-1}\Id_{h_1}),
		$$
		if n is even, where $h_i\in \mathbb{N}$ and  $h_r\geq 0$ even, and 
		$$
		z=\diag(z_1\Id_{h_1},\dots, z_{r-1}\Id_{h_{r-1}},\Id_{h_{r}},z_{r-1}^{-1}\Id_{h_{r-1}},\dots, z_1^{-1}\Id_{h_1}),
		$$
		if n is odd, where $h_i\in \mathbb{N}$ and $h_r\geq 1$ odd.\\
		Therefore, looking at $(T(x)t)^{\reg}$, for $n$ even 
		$$
		z \text{ is such that } \ z_i\neq z_j^{\pm 1} \text{ if } i\neq j, \ z_i\neq 1, \ h_i\in \mathbb{N}, \ h_r\geq0 \text{ even }, $$
		and for $n$ odd 
		$$
		z  \text{ is such that } z_i\neq z_j^{\pm 1} \text{ if } i\neq j, \ z_i\neq 1,  \ h_i\in \mathbb{N}, \ h_r\geq1 \text{ odd }. $$
		Then 
		$J(x)=SL(n)\cdot ((T(x)t)^{\reg}u\tau)$.
		
		Lastly, as we saw in Remark \ref{unipotent jordan class}, the class $J(u\tau)$ with $u\tau$ unipotent is the orbit $G^{\circ}\cdot u\tau$.

		\subsection{The case $n=3$}
		
		In this section we analize the case $n=3$ and compute explicitly the poset of Jordan classes with respect to the partial order defined in Subsection \ref{poset}. \\

		Every element $x'\in SL(3)\tau$ is $SL(3)$-conjugate to an element $x=tu\tau$ with $t\in T^{\tau}$ and $u\in U$ such that $tu=ut$. 
		Observe that a non-trivial element in $T^{\tau}$ is always regular when $n=3$. Since if $t$ is regular then $u=\Id$, we have only two possibilities for $x$, namely
		$x=t\tau$ or $x=u\tau$.

		Let $x=t\tau$ with $t\in T^{\tau}$. 
		Then, as we have seen in the general case, 
		\begin{equation}\label{ttau}
			\tag{*}
			J(t\tau)=SL(3)\cdot \left(\left\{\begin{pmatrix}a&&\\&1&\\&&a^{-1} \end{pmatrix} \ \Bigg| \ a\neq 0,1\right\}\tau\right).
		\end{equation}
		
		We analize the isolated strata associated to $J(t\tau)$.\\ The torus $T(t\tau)=T^{\tau}$, so in this case
		$L(t\tau)=C_{SL(3)}(T(t\tau))=T$,
		$
		N_{G}T=T\rtimes \langle \tau \rangle
		$,
		$
		P=B,
		$
		$
		N_GP=B\rtimes \langle \tau \rangle.
		$
		Thus the isolated stratum $S$ of $t\tau$ in $N_GL$ is 
		$$
		S=T(T\cdot t\tau)=T\tau.
		$$
		Let $y\tau=\begin{pmatrix}\alpha&&\\&\beta &\\ &&\alpha^{-1}\beta^{-1}\end{pmatrix}\tau\in T\tau$. Then its Jordan decomposition is
		$$
		(y\tau)_s=\begin{pmatrix}\alpha\sqrt{\beta}&&\\&1&\\&&\alpha^{-1}\sqrt{\beta}^{-1}\end{pmatrix} \text{ and }(y\tau)_u=\begin{pmatrix}\sqrt{\beta}^{-1}&&\\&\beta&\\ &&\sqrt{\beta}^{-1}\end{pmatrix}\tau.
		$$
		Thus the set $S^{*}=\{y\tau\in T\tau \ | \ C_G((y\tau)_s)^{\circ}\subseteq T\}$ is 
		$$	
		\left\{\begin{pmatrix}\alpha&&\\&\beta&\\&&\alpha^{-1}\beta^{-1}\end{pmatrix} \tau \quad \Bigg | \quad \alpha,\beta\in \mathbb{K}^*, \ \alpha\neq\sqrt{\beta}^{-1}\right\}.
		$$
		
		\phantom{aaa}\\
		
		We consider now $x=u\tau$, with $u\in U$.
		By Remark \ref{unipotent jordan class}, the torus $T(u\tau)=\{\Id\}$, so the Jordan class of $x$ is just the $G^{\circ}$-orbit of $u\tau$.
		In this case the Jordan class is isolated.
		
		The following Lemma lists all the possible Jordan classes $J(x)$ in $G\tau$ that are unipotent $G^{\circ}$-orbits.

		\begin{lemma}
			There are only two unipotent $SL(3)$-orbits represented in $U\tau$:\\
			\begin{itemize}
				\item the orbit of $\tau$, 
				\item the orbit $SL(3)\cdot (u_1\tau)$ where $u_1=\begin{pmatrix}
					1&1&0\\0&1&0\\0&0&1
				\end{pmatrix}$.
			\end{itemize}
			
		\end{lemma}
		\begin{proof}
			Let $y\in G^{\circ}$ be the matrix with entries $y_{i,j}$ and such that $y\tau\in G^{\circ}\cdot\tau$, so there exists $x\in G^{\circ}$ such that $y=x\tau(x^{-1})$. Computing the product $x\tau(x^{-1})$, we find that $y_{1,2}=y_{2,3}$, giving a necessary condition for $y\tau$ to lie in the $G^{\circ}$-orbit of $\tau$. \\
			Let $u_1=\begin{pmatrix}
				1&1&0\\0&1&0\\0&0&1
			\end{pmatrix}$, then $u_1\tau\notin G^{\circ}\cdot \tau$.\\
			
			Let $\overline a=\begin{pmatrix}
				1&a_1&a_2\\0&1&a_3\\0&0&1
			\end{pmatrix}\in U$.
			We distinguish two cases.
			\begin{itemize}
				\item if $a_1=a_3$, it possible to find a matrix $\overline{b}\in U$ such that $\overline b\overline a\tau(\overline b^{-1})=\Id$. So $\overline a\tau \in G^{\circ}\cdot\tau$.
				\item if $a_1\neq a_3$, there exists $\overline b'\in U$ such that $\overline b'\overline a\tau(\overline b'^{-1})=u_1$. Then $\overline a\tau \in G^{\circ}\cdot u_1\tau$.
			\end{itemize}
			Therefore the only unipotent orbits represented in $U\tau$ are $G^{\circ}\cdot\tau$ and $G^{\circ}\cdot u_1\tau$.

		\end{proof}
		The next theorem summarizes the results of this section.
		\begin{teorema}
			The Jordan classes of $SL(3)\rtimes\langle \tau \rangle$ are 
			\begin{enumerate}[-]
				\item $J(t\tau)$ as in (\ref{ttau}) , with $t\in T^{\tau}$, $t\neq \Id$;
				\item $J(u_1\tau)=SL(3)\cdot u_1\tau$, where $u_1=\begin{pmatrix}
					1&1&0\\0&1&0\\0&0&1
				\end{pmatrix}$;
				\item $J(\tau)=SL(3)\cdot \tau$.
			\end{enumerate}
		\end{teorema}
		
		\phantom{aaa}\\
		We analize the poset of the Jordan classes.\\
		By \cite[II 2.21]{Spaltestein}, there exists a unique minimal unipotent quasi-semisimple orbit in $G^{\circ}\tau$, that is $G^{\circ}\cdot \tau$.\\
		Thus we have $\dim G^{\circ}\cdot u_1\tau>\dim G^{\circ}\cdot \tau$. So $\tau\notin G^{\circ}\cdot (\tau U)^{\reg}$, then the unipotent induced orbit $\Ind_{T\tau}^{G^{\circ}\tau}(T\cdot \tau)$ is $G^{\circ}\cdot u_1\tau$. By Proposition \ref{reg}, $\Ind_{T\tau}^{G^{\circ}\tau}(T\cdot \tau)\subset \overline{J(t\tau)}^{\reg}$, so $J(u_1\tau)=G^{\circ}\cdot u_1\tau\subset \overline{J(t\tau)}^{\reg}$. Therefore
		\begin{itemize}
			\item $
			\overline{J(t\tau)}^{\reg}=J(t\tau)\cup J(u_1\tau)$ and $
			\overline{J(t\tau)}=J(t\tau)\cup J(u_1\tau)\cup J(\tau)=G^{\circ}\tau;
			$
			\item $\overline{J(u_1\tau)}^{\reg}=J(u_1\tau)=G^{\circ}\cdot u_1\tau$;
			\item $\overline{J(\tau)}^{\reg}=J(\tau)=G^{\circ}\cdot \tau$.
		\end{itemize}
		Then the only confrontable elements are $J(u_1\tau)\leq J(t\tau)$.\\
		
		Now we look at the Lusztig strata in $G$. By Theorem \ref{strata} any stratum is a union of regular closures of Jordan classes whose elements have $G^{\circ}$-orbits of the same dimension. 
		Thus there are two strata 
		$$
		X_1=\overline{J(t\tau)}^{\reg} \text{ and } X_2=\overline{J(\tau)}^{\reg}.
		$$
		The reader can compare this result with the example in \cite[3.4]{lusztig2020strata}.
	\end{section}
	\section*{Acknowledgments}
	It is a pleasure to thank Giovanna Carnovale and Francesco Esposito for proposing the problem and for helpful suggestions. The author would like to thank Mauro Costantini for his help.
	The author acknowledges support by:  DOR2207212/22 ``Algebre di Nichols, algebre di Hopf e gruppi algebrici" and  BIRD203834 ``Grassmannians, flag varieties and their generalizations." funded by the University of Padova. She is a member of the  INdAM group GNSAGA.
	

\begin{thebibliography}{10}
			\bibitem{carno}
		G.~Carnovale.
		\newblock Lusztig's strata are locally closed.
		\newblock {\em Archiv der Mathematik}, 115, 03 2020.
		
			\bibitem{CarnovaleEsposito}
		G.~Carnovale and F.~Esposito.
		\newblock On sheets of conjugacy classes in good characteristic.
		\newblock {\em International Mathematics Research Notices}, 2012:810--828,
		2012.
	
		\bibitem{dignemichel}
	F.~Digne and J.~Michel.
	\newblock Groupes r\'eductifs non connexes.
	\newblock {\em Annales scientifiques de l'\'Ecole Normale Sup\'erieure}, 4e
	s{\'e}rie, 27(3):345--406, 1994.
	
		\bibitem{humphreys}
	{J}.~{E}. {H}umphreys.
	\newblock {\em Linear Algebraic Groups}.
	\newblock Springer New York, NY, 1975.
	
			\bibitem{outerunipotent}
	R.~Lawther, M.~W. Liebeck, and G.~M. Seitz.
	\newblock Outer unipotent classes in automorphism groups of simple algebraic
	groups.
	\newblock {\em Proceedings of the London Mathematical Society},
	109(3):553--595, 2014.
	
		\bibitem{intersectioncohom}
	G.~Lusztig.
	\newblock Intersection cohomology complexes on a reductive group.
	\newblock {\em Inventiones mathematicae,75}, pages 205--272, 1984.
	
		\bibitem{charactersheafI}
	G.~Lusztig.
	\newblock Character sheaves on disconnected groups {I}.
	\newblock {\em Represent. Th. 7}, pages 374--403, 2003.
	
			
	\bibitem{Lusztig2003HeckeAW}
	G.~Lusztig.
	\newblock Hecke algebras with unequal parameters.
	\newblock American Mathematical Soc., 2003.
	
		
	\bibitem{charactersheafII}
	G.~Lusztig.
	\newblock Character sheaves on disconnected groups {II}.
	\newblock {\em Represent. Theory 10 (2006)}, pages 353--379, 2006.
	
			\bibitem{Lusztig2015}
	G.~Lusztig.
	\newblock {\em On conjugacy classes in a reductive group}, pages 333--363.
	\newblock Springer International Publishing, Cham, 2015.
	
		\bibitem{lusztig2020strata}
		G.~Lusztig.
		\newblock Strata of a disconnected reductive group.
		\newblock {\em Indagationes Mathematicae}, 32(5):968--986, 2021.
		\newblock Special issue to the memory of T.A. Springer.
		
	\bibitem{LS}
G.~Lusztig and N.~Spaltenstein.
\newblock Induced unipotent classes.
\newblock {\em Journal of the London Mathematical Society}, s2-19(1):41--52,
1979.
		
		\bibitem{Sorlin}
	G.~Malle and K.~Sorlin.
	\newblock Springer correspondence for disconnected groups.
	\newblock {\em Mathematische Zeitschrift}, 246:291--319, 01 2004.
	
	\bibitem{Spaltestein}
	N.~Spaltestein.
	\newblock {\em Classes Unipotentes et Sous-groupes de Borel}.
	\newblock Springer-Verlag Berlin Heidelberg, 1982.
	
		\bibitem{twistedspringer}
T.~Springer.
\newblock Twisted conjugacy in simply connected groups.
\newblock {\em Transformation Groups}, 11:539--545, 2006.

	\bibitem{PMIHES_1965__25__49_0}
	R.~Steinberg.
	\newblock Regular elements of semi-simple algebraic groups.
	\newblock {\em Publications Math\'ematiques de l'IH\'ES}, 25:49--80, 1965.
	
		
	\bibitem{steinberg}
	R.~Steinberg.
	\newblock {\em Endomorphisms of Linear Algebraic Groups}.
	\newblock American Mathematical Soc., 1968.
	
		\bibitem{springer}
		Springer T.A.
		\newblock {\em Linear algebraic groups}.
		\newblock 1998.
		

	\end{thebibliography}
\end{document}